\documentclass{amsart}
\usepackage{amsmath,amsfonts,amssymb,amscd,verbatim,multicol}
\usepackage{fullpage}

\newcommand{\blank}{\mbox{$\underline{\makebox[10pt]{}}$}}

\newtheorem{theorem}{Theorem}[section]
\newtheorem{lemma}[theorem]{Lemma}

\newtheorem{corollary}[theorem]{Corollary}
\theoremstyle{definition}
\newtheorem{example}[theorem]{Example}

\newtheorem{remark}[theorem]{Remark}

\numberwithin{equation}{theorem}

\DeclareMathOperator{\Ext}{Ext}
\DeclareMathOperator{\Tor}{Tor}

\DeclareMathOperator{\uTor}{\underline{Tor}}
\DeclareMathOperator{\gl}{gldim}

\newcommand{\wt}{\widetilde}
\DeclareMathOperator{\im}{im}

\newcommand{\mb}{\mathbb}
\newcommand{\mc}{\mathcal}
\DeclareMathOperator{\ddeg}{deg}

\copyrightinfo{2011}{American Mathematical Society}

\begin{document}

\title{Regular algebras of dimension 4 with 3 generators}

\author{D. Rogalski and J. J. Zhang}

\begin{abstract}
We study Artin-Schelter regular algebras of global dimension 4 with
three generators of degree one.  We classify those which are domains
and which have an additional $\mb{Z}^{\times 2}$-grading, and prove
that all of these examples are also strongly noetherian, Auslander
regular, and Cohen-Macaulay.
\end{abstract}


\address{(Rogalski)
UCSD Department of Mathematics, 9500 Gilman Dr. \# 0112, La Jolla,
CA 92093-0112, USA. } \email{drogalsk@math.ucsd.edu}

\address{(Zhang)
University of Washington, Department of Mathematics, Box 354350,
Seattle, WA 98195-4350, USA. } \email{zhang@math.washington.edu}


\maketitle

\setcounter{section}{-1}
\section{\bf Introduction}
\label{xxsec0}

In the late 1980's, Artin and Schelter introduced in \cite{ASc} the
notion of a \emph{regular algebra} (now often called an
Artin-Schelter regular or AS-regular algebra.)  By definition, an
AS-regular algebra is an $\mb{N}$-graded algebra $A = \bigoplus_{n
\geq 0} A_n$ over a field $k$ which  is connected ($A_0 = k$) and
satisfies the following three conditions: (i) $A$ has finite global
dimension $d$; (ii) $A$ has polynomial growth; and (iii) $A$ is
Gorenstein in the sense that $\Ext^i_A(k, A) = 0$ for $i \neq d$ and
$\Ext^d_A(k, A) \cong k$.  (This last condition is now also called
AS-Gorenstein.) Artin and Schelter's original work, together with
subsequent work of Artin, Tate, and Van den Bergh, led to a complete
classification of AS-regular algebras of global dimension 3, or
equivalently, quantum ${\mathbb P}^2$s \cite{ASc, ATV1, ATV2}.

Since this seminal work there has been extensive research about
AS-regular algebras of dimension four.  The dimension of an
AS-regular algebra means its global dimension. An important and
famous example is the Sklyanin algebra of dimension 4, introduced by
Sklyanin \cite{Sk1, Sk2} and studied further by Smith and Stafford
\cite{Sm-Sta}, among others. Some other examples of well-known
AS-regular algebras of dimension four include graded regular
Clifford and skew-Clifford algebras \cite{CV}; homogenized
$U({\mathfrak {sl}}_2)$ \cite{Le-Sm2}; deformations of the Sklyanin
algebra \cite{Sta94b}; skew polynomial rings in 4 variables; the
quantum $2\times 2$-matrix algebra \cite{Vanc1}; AS-regular algebras
with finitely many points; and AS-regular algebras containing a
commutative quadric \cite{Stephenson-Vanc, SVanc1, SVanc2, VVanc1,
VVanc2, VVW}. Recently, Jun Zhang and the second author introduced
the notion of a double Ore extension, which is also a useful method
for constructing AS-regular algebras of dimension four \cite{ZZ1,
ZZ2}.

For simplicity, in this paper we only consider graded algebras that are
generated in degree 1. By \cite{LPWZ}, every noetherian AS-regular algebra
$A$ of dimension four is generated by either 2, or 3, or 4 elements.
If $A$ is generated by 4 elements, then it is Koszul and its Hilbert
series is equal to $(1-t)^{-4}$ (the same as the Hilbert series of
the commutative polynomial ring $k[x_1,x_2,x_3,x_4]$). All algebras
listed in the previous paragraph are generated by 4 elements.  More
recently, Lu, Palmieri, Wu, and the second author  have
investigated AS-regular algebras of dimension four that are generated
by 2 elements \cite{LPWZ}.

The main objective of this paper is to begin to study AS-regular
algebras $A$ of dimension four that are generated by 3 elements.  It
is easy to find some examples of these by taking Ore extensions of
$2$-generated AS-regular algebras of dimension three, but few if any
more general kinds of examples appear to be known. The idea of this
paper is to look for some examples with the added special property
that they have an additional $\mathbb{Z}^{\times 2}$-grading which
is proper in the sense that $A_1=A_{1,0}\oplus A_{0,1}$ with
$A_{1,0} \neq 0$ and $A_{0,1} \neq 0$. We call those $A$ which also
have a normal regular element $x$ of degree $1$ such that $A/(x)$ is
AS-regular of dimension three \emph{normal extensions}. As normal
extensions are in general easier to produce and to understand, we do
not study them in detail here. By using some fairly Na\"ive
techniques (primarily Bergman's diamond lemma \cite{Be}) and some
simple code in the mathematical software program Maple to aid with
the computation, we prove the following classification result.

\begin{theorem}
\label{xxthm0.1} Let $A$ be an AS-regular domain of dimension 4
which is generated by $3$ elements and properly $\mb{Z}^{\times
2}$-graded. Then either $A$ is a normal extension, or else up to
isomorphism $A$ falls into one of eight $1$ or $2$-parameter
families of examples (which we give explicitly in
Section~\ref{xxsec3} below.)
\end{theorem}

Since a noetherian AS-regular algebra of dimension 4 is a domain
\cite[Corollary 1.5]{LPWZ}, the above theorem also gives a
classification of noetherian AS-regular algebras of dimension 4
which are generated by $3$ elements and properly $\mb{Z}^{\times
2}$-graded. These eight families of algebras in Theorem
\ref{xxthm0.1} are generically pairwise non-isomorphic [Theorem
\ref{xxthm5.2}(b)]. The known examples of AS-regular algebras of
dimension 4 all have several other good ring-theoretic and
homological properties which conjecturally may follow automatically
from the AS-regularity assumption.  One of our motivations is to
continue to provide evidence for such a conjecture by showing that
all of our examples have these properties.  The following result
follows from our explicit classification in Theorem~\ref{xxthm0.1}.

\begin{theorem}
\label{xxthm0.2} Let $A$ be a noetherian AS-regular algebra of dimension
4 that is generated by 3 elements. If $A$ is properly
$\mb{Z}^{\times 2}$-graded, then it is strongly noetherian, Auslander
regular and Cohen-Macaulay.
\end{theorem}
Combining this theorem with the results in \cite{LPWZ, ZZ2}, we have
the following corollary.

\begin{corollary}
\label{xxcor0.3}
Let $A$ be noetherian AS-regular algebra of dimension 4. If $A$ is
properly $\mb{Z}^{\times 2}$-graded, then it is strongly noetherian,
Auslander regular and Cohen-Macaulay.
\end{corollary}

Another motivation for our paper is to provide some new families of
examples to serve as a testing ground for future questions and
conjectures about AS-regular algebras.  For example, Kirkman, Kuzmanovich,
and the second author have developed a theory of rings of invariants
of finite group actions on AS-regular algebras, concentrating mostly on
the Koszul case \cite{KKZ1,KKZ2,KKZ3}, and having more examples at
hand may be useful in extending this theory to the non-Koszul case.
For this reason, we have included in our paper a careful determination
of the automorphism groups of the generic examples in our classification
[Theorem \ref{xxthm5.2}(a)].

\bigskip

{\bf Acknowledgments:}\;  Both authors thank the NSF for support,
and also thank Paul Smith for helpful conversations and for allowing
them to include in Lemma~\ref{xxlem1.1} an unpublished result proved
by him and the second author.

\section{\bf Preliminaries}
\label{xxsec1}
In this section, we give some background and lemmas that will be of use
in our study of $\mb{Z}^{\times 2}$-graded AS-regular algebras with $3$
generators.  We work always over an algebraically closed field $k$ of
characteristic $0$.  All the algebras $A$ in this paper will be
connected graded, so $A = k \oplus A_1 \oplus A_2 \oplus \dots$, and
generated as a $k$-algebra by $A_1$. We write $k$ for the trivial graded
$A$-module $(A/A_{\geq 1})$ (as a left or right module, depending on
context). The degree shift of a $\mb{Z}$-graded $A$ module $M$ is
written as $M(d)$, where $M(d)_n = M_{n+d}$ for all $n$.  If
$\sigma: A \to A$ is a graded automorphism, recall that the
\emph{graded twist} of $A$ by $\sigma$ is the
new algebra $A^\sigma$ with the same underlying vector space as $A$ and a new
product defined by $x \star y = x\sigma^m(y)$, for $x \in A_m, y \in A_n$
(see \cite{Zh2} for details). By \cite[Theorem 1.3]{Zh2}, $A$ and
$A^\sigma$ share many homological properties.

There are many standard techniques for relating the properties of a
graded algebra $A$ and a factor algebra $A/(x)$, where $x$ is a
homogeneous regular normal element, and these will play a large role
in our analysis.  We begin with a general such result that may be of
independent interest.

\begin{lemma}
\label{xxlem1.1}  \cite{SmZh}
Let $B$ be a connected graded algebra of global dimension $d<\infty$
and let $z\in B_1$ be a regular normal element of degree 1. Then
$B/(z)$ has global dimension $d-1$.
\end{lemma}

\begin{proof}
Since $B$ is connected graded, $\gl B = \max\{i\;|\;
\uTor_i^B(k,k)\neq 0\}$, where $\uTor^B_i$ is the graded version of
the $\Tor$ functor. Let $C=B/(z)$. We have the following graded
version of the spectral sequence \cite[Theorem 10.59]{Rot}:
\begin{equation}
\label{E1.1.1}\tag{E1.1.1}
\uTor^{C}_p(k,\uTor^{B}_q(C,k))\quad \Longrightarrow_p \quad \uTor^{B}_n(k,k).
\end{equation}
Since $z$ is a regular element, by applying $\blank \otimes_B k$ to the
short exact sequence $0 \to zB \to B \to C \to 0$ and considering the
associated long exact sequence, we easily see that $\uTor_0^B(C,k)=k$,
$\uTor_1^B(C,k)=k(-1)$, and $\uTor_i^B(C,k)=0$ for $i>1$.  Hence the
$E^2$-page of the spectral sequence \eqref{E1.1.1} has only two
possibly non-zero rows; namely
$$
q=0:\quad \uTor^{C}_p(k,k)\ \text{for}\ p \geq 0,\
\qquad \text{and}\ \qquad q=1:\quad \uTor^{C}_p(k, k(-1))\
\text{for}\ p \geq 0.
$$
Since \eqref{E1.1.1} converges, we have $\uTor^C_0(k,k)=\uTor^B_0(k,k)$
and a long exact sequence
\begin{align}
\cdots \cdots &\longrightarrow \uTor^C_4(k,k)\longrightarrow
\uTor^C_2(k,k(-1)) \notag\\
\longrightarrow \uTor^B_3(k,k)&\longrightarrow \uTor^C_3(k,k)
\longrightarrow \uTor^C_1(k,k(-1))
\label{E1.1.2}\tag{E1.1.2}\\
\longrightarrow \uTor^B_2(k,k)&\longrightarrow \uTor^C_2(k,k)
\longrightarrow \uTor^C_0(k,k(-1)) \notag\\
\longrightarrow \uTor^B_1(k,k)&\longrightarrow \uTor^C_1(k,k)
\longrightarrow 0. \notag
\end{align}
By hypothesis, $d=\gl B <\infty$. By \eqref{E1.1.2}, $\uTor^C_{i+1}(k,k)
\cong \uTor^C_{i-1}(k,k)(-1)$ for all $i>d$. Suppose for some  $i > d$
that $\uTor^C_{i+1}(k,k) \neq 0$.  Then the minimal nonzero degree of
$\uTor^C_{i+1}(k,k)$ is equal to one more than the minimal nonzero
degree of $\uTor^C_{i-1}(k,k)$. On the other hand, note that the
$i^{th}$ term of the minimal free resolution of the trivial $C$-module
$k_C$ is isomorphic to $\uTor^C_{i}(k,k)\otimes_k C$. By the
minimality of this free resolution, the minimal nonzero degree of
$\uTor^C_{i+1}(k,k)$ is at least one bigger than the minimal nonzero
degree of $\uTor^C_{i}(k,k)$, hence is at least two bigger than the
minimal nonzero degree of $\uTor^C_{i-1}(k,k)$.  This is a
contradiction, and thus for $i > d$, we must have
$\uTor^C_{i+1}(k,k) = 0$. Thus $\gl C<\infty$.  By
\cite[Lemma 7.6]{LPWZ}, $\gl B=\gl C +1$.
\end{proof}

The converse of the above lemma is also true, and in fact does not
require the normal element to be in degree $1$.  The reader can find
a proof in \cite[Lemma 7.6]{LPWZ}, or may easily prove the converse
using the same spectral sequence as in the lemma.

\begin{corollary}
\label{xxcor1.2}
Let $A$ be a connected graded algebra which is a domain.  If there is
a nonzero normal element $x \in A_1$, then $A/(x)$ is AS-regular if
and only if $A$ is.
\end{corollary}

\begin{proof}
The preceding lemma and the converse in \cite[Lemma 7.6]{LPWZ} show
that $A$ has finite global dimension if and only if $A/(x)$ does.
It is trivial that $A$ has polynomial growth if and only if $A/(x)$
does, and that $A$ is AS-Gorenstein if and only if $A/(x)$ is follows
from the Rees Lemma, which says that $\Ext^i_A(k, A) \cong
\Ext^{i-1}_{A/(x)}(k, A/(x))$.
\end{proof}

Besides AS-regularity, some other important properties of graded
algebras are the strong noetherian property, Auslander regularity,
and the Cohen-Macaulay property.  We will want to show that the
examples we construct in this paper have these properties, but
otherwise will not use them, so we refer the reader to \cite[Section
5]{ZZ1} for a review of the definitions.   In the next result we
review various techniques we will use to prove our examples have
good properties.

\begin{lemma}
\label{xxlem1.3}
Let $A$ be a connected graded algebra.   If any of the following
conditions holds, then $A$ is an AS-regular algebra
which is also strongly noetherian, Auslander regular and Cohen-Macaulay.
\begin{enumerate}
\item
$A \cong B[t;\sigma,\delta]$ for an AS-regular algebra $B$ of dimension 3
with automorphism $\sigma$ and $\sigma$-derivation $\delta$.
\item
$A$ is a graded twist of $B$, where $B$ is strongly noetherian,
Auslander regular, and Cohen-Macaulay.
\item
$f$ is a regular homogeneous normal element and $B:=A/(f)$ is
AS-regular of dimension 3.
\item
$A$ has finite global dimension, and there are elements
$f_1,\cdots, f_t$ such that the image of $f_i$ is
normal (not necessarily regular) in the factor ring
$A/(f_1,\cdots,f_{i-1})$ for all $i$, and $A/(f_1,\cdots f_t)$
is finite dimensional over $k$.
\end{enumerate}
\end{lemma}

\begin{proof}
(a) This is \cite[Lemma 5.3]{ZZ1}.

(b) This follows from \cite[Lemma 5.6]{ZZ1}.

(c,d) These follow from \cite[Lemma 5.8]{ZZ1}, \cite[Remark 5.2(b)]{ZZ1}
and Lemma \ref{xxlem1.1}. Part (d) is also a consequence of
\cite[Theorem 1]{Zh1}.
\end{proof}

\section{\bf Beginning analysis of $\mb{Z}^{\times 2}$-graded AS-regular
algebras}
\label{xxsec2}

Starting in this section, let $A$ stand for an AS-regular algebra of
dimension $4$ with three generators of degree $1$ which is properly
$\mb{Z}^{\times 2}$-graded.  Without loss of generality, we will
write $A_1 = kx_1 + kx_2 + kx_3$ where $\ddeg x_1 =\ddeg x_2 =
(1,0)$ and $\ddeg x_3 = (0,1)$.  By \cite[Proposition 1.4(b)]{LPWZ},
the Hilbert series of $A$ as a connected $\mathbb{N}$-graded algebra
is $\displaystyle h_A(t) = \frac{1}{(1-t)^3(1-t^2)}$.  We always
assume in this paper that $A$ is a domain. (It is not known if this
holds automatically for AS-regular algebras.)  The assumptions on
$A$ in this paragraph are fixed for the rest of the paper.

Recall that an algebra $A$ as above is a \emph{normal extension} (of
an AS-regular algebra of dimension three) if there is a nonzero normal
element $z\in A_1$. In this case $A/(z)$ is $3$-dimensional AS-regular
algebra by Corollary~\ref{xxcor1.2}, and so most homological and
ring-theoretic properties of $A$ follow automatically from those
of $A/(z)$.  Normal extensions are in general easier to understand;
in the case of AS-regular algebras with $4$-generators, for example,
see \cite{Le-Sm-VdB} for an extensive study of them.  Also, we do
not need to know any details about those $A$ which are normal
extensions in order to prove Theorem~\ref{xxthm0.2}. Thus for
simplicity we will ignore normal extensions and classify only those
$A$ without normal elements in degree $1$.

Since $A$ is $\mb{Z}^{\times 2}$-graded, all the relations of $A$
are also $\mb{Z}^{\times 2}$-homogeneous. In the next result we
determine the multi-degrees of these relations.

\begin{lemma}
\label{xxlem2.1}
Retain the hypotheses for $A$.  Suppose that $A$ has no normal
element in degree 1.
\begin{enumerate}
\item
$A$ has a relation in degree $(2,0)$ and a relation in degree $(1,1)$.
\item
$A$ has a relation in degree $(1,2)$ and a relation in degree $(2,1)$.
\item
The $\mb{Z}^{\times 2}$-graded Hilbert series of $A$ equals
\begin{equation}
\label{E2.1.1} \tag{E2.1.1} h_A(u,v) = \frac{1}{(1-u)^2(1-v)(1-uv)},
\end{equation}
where $u$ has degree $(1,0)$ and $v$ has degree $(0,1)$.
\end{enumerate}
\end{lemma}

\begin{proof}
By \cite[Proposition 1.4(b)]{LPWZ}, $A$ has two relations $r_1, r_2$
in degree 2 and two relations $r_3, r_4$ in degree 3.  If there is a
relation in degree $(0,2)$, then it must be $x_3^2=0$, which
contradicts the fact $A$ is a domain. If there are two relations in
degree $(1,1)$, then using the fact $A$ is a domain, after replacing
them by linear combinations we can assume the two relations have the
form
$$x_3 x_1= a x_1 x_3+bx_2 x_3 \quad {\rm and}\quad x_3x_2=c x_1
x_3+d x_2 x_3.$$
Thus $x_3$ will be a normal element, contradicting the hypothesis.
By \cite[Lemma 3.7(a)]{KKZ1}, $A$ cannot have two relations in
degree $(2,0)$.  Thus $A$ has one relation in degree $(2,0)$ and
one relation in degree $(1,1)$, proving (a).

Consider now the two relations of degree $3$.  By
\cite[Lemma 3.7(a)]{KKZ1}, $A$ cannot have relations in degree $(3,0)$,
and since $A$ is a domain, it has no relation of degree $(0,3)$.
Thus the remaining relations have degree equal to $(2,1)$ or $(1,2)$.
Let $E$ be the $\Ext$-algebra of $A$.  Then
$$E=k\oplus E_{-1}^{\oplus 3}\oplus (E_{-2}^{\oplus 2}\oplus
E_{-3}^{\oplus 2})\oplus E_{-4}^{\oplus 3} \oplus E_{-5}$$
is in fact $\mb{Z}\times (\mb{Z}^{\times 2})$-graded. Using the
language introduced in \cite[Proposition 3.1(b)]{LPWZ},
$\{r_1^*, r_2^*\}$ is a basis of $E^2_{-2}$ and $\{r_3^*,r_4^*\}$
is a basis of $E^2_{-3}$. Since $E$ is Frobenius
\cite[Theorem 1.9]{LPWZ}, the perfect pairing on the
$\mb{Z}\times (\mb{Z}^{\times 2})$-graded algebra $E$ implies that
$E_{-3}^2$ has one element of degree $(2,1)$ and one element of
degree $(1,2)$. This means that $A$ has exactly one relation of
each degree $(2,1)$ and $(1,2)$.  Thus we have proved (b), and
furthermore we have determined the multi-grading on $E$.

Finally, the multi-graded Hilbert series of $A$ can be computed by
using the multi-graded version of the free resolution in
\cite[Proposition 1.4(b)]{LPWZ}.
\end{proof}

We remark that if $x_3$ is a normal element in $A$ (the case we
threw away in the proof of the lemma), then the multi-graded
Hilbert series of $A$ is different from \eqref{E2.1.1}.

\begin{corollary}
\label{xxcor2.2}
If $A$ has no normal nonzero element of degree $1$, then the
subalgebra $C$ of $A$ generated by $x_1$ and $x_2$ is a
$2$-dimensional AS-regular algebra.
\end{corollary}

\begin{proof} Since $A$ is a domain, so is $C$. And $C$ has at
least one degree $2$ relation, as we saw in the proof of the lemma.
The assertion follows from \cite[Lemma 3.7(a)]{KKZ1}.
\end{proof}

In the remainder of this section, we begin to analyze the possible
relations $r_1-r_4$ of non-normal-extensions $A$, and do some
reductions to simply the form of the relations.  By
Lemma~\ref{xxlem2.1}, without loss of generality from now on we
can assume that
\begin{equation}
\label{E1.2.1}\tag{E1.2.1} \ddeg r_1=(2,0),\; \ddeg r_2=(1,1),\;
\ddeg r_3=(1,2),\; {\rm and }\; \ddeg r_4=(2,1).
\end{equation}
Throughout our analysis, we will always let monomials be ordered
with degree-lex order with $x_3> x_2 > x_1$.  Note that the opposite
ring of $A$ again satisfies the same hypotheses as $A$. One of our
main tools in this paper will be Bergman's diamond lemma \cite{Be},
with which we will assume the reader is familiar. We will sometimes
pass to the opposite ring in order to simplify the relations.  Since
we will calculate the opposite rings of the examples we produce, we
will not miss any examples this way.

The single relation of a $2$-dimensional AS-regular algebra is
well-known to be expressible in one of two forms. Since
$C = k \langle x_1, x_2 \rangle/(r_1)$ is AS-regular by the
corollary above, by making a $k$-linear base change in $kx_1+kx_2$,
we can assume that $r_1$ is either a quantum-type relation
$x_2x_1 = p x_1 x_2$ for some $p \in k^{\times}$, or a Jordan-type
relation $x_2x_1 = x_1x_2 + x_1^2$ (see the proof of \cite[Lemma 3.7]{KKZ1}.)
We write $r_1$ in a general form allowing both possibilities as follows:
$$r_1: \qquad x_2x_1=px_1x_2 + m x_1^2,\
\text{where either}\ p \in k^{\times}\ \text{and}\ m = 0,\
\text{or}\ p = m = 1.$$

Next, $r_2$ has the general form $x_3(dx_2 +ax_1) + (n x_1 +bx_2)x_3 = 0$.
First we claim that we may take $d = 1$. If $m = 0$, so that $r_1$ is of
quantum type, then since $A$ is a domain we have $(dx_2 + ax_1) \neq 0$.
Then switching the roles of $x_1$ and $x_2$ if necessary, which does
not affect the form of $r_1$, we may assume that $d \neq 0$.  If $r_1$
is of Jordan type, so $m = p = 1$, then if $d = b= 0$ then $x_1$ will
be normal in $A$.  Since we are excluding normal extensions, by passing
to the opposite ring if necessary, we can assume that $d \neq 0$.
(Note that after passing to the opposite ring, we can put $r_1$ in the
same form as before by replacing $x_1$ by $-x_1$.)  Thus in all cases
we may write the relation in the form
$$ r_2: \qquad x_3 x_2 = ax_3 x_1 + n x_1x_3 + b x_2x_3.  $$
Further simplifications to this relation are possible here, depending
on the case for $r_1$. We will make these simplifications in our case
by case analysis later.

There is an overlap ambiguity $x_3x_2x_1$ between $r_1$ and $r_2$.
Resolving this gives a new relation
$$r_5: \qquad x_3x_1x_2 =qn x_1x_3x_1 + qb x_2x_3x_1 + z x_3x_1^2,$$
where we write $q=p^{-1}$ and $z = q a-qm$ for notational convenience.

Since we have relations with leading terms $x_2x_1$ and $x_3x_2$,
the degree $(1,2)$ relation $r_3$ can be put into the form
$(cx_1+dx_2)x_3^2+ex_3x_1x_3+tx_3^2x_1 = 0$.  If $t =0$, then as $A$
is a domain, we must have $(cx_1 +dx_2)x_3 + ex_3x_1=0$.  Since
$r_2$ is the only degree $(1,1)$-relation by Lemma~\ref{xxlem2.1},
this must be a scalar multiple of $r_2$, which it clearly is not. So
$t \neq 0$, and the relation $r_3$ can be assumed to have the form
$$r_3: \qquad x_3^2x_1 = c x_1x_3^2 + d x_2 x_3^2 + e x_3 x_1 x_3. $$

By Lemma \ref{xxlem2.1}, there is one more relation $r_4$ of degree
$(2,1)$. Given the leading terms of the relations we already have,
$r_4$ can be written as
$$r_4: \qquad (f x_1^2 + g x_1 x_2 + h x_2^2 )x_3 +
(j x_1 + k x_2) x_3 x_1 + \ell x_3 x_1^2 = 0.$$
The leading term of $r_4$ is not clear at this point.  We will do a
separate analysis of each possible case of the leading term of $r_4$,
but the only case in which we find any AS-regular algebras will be
when $\ell \neq 0$.

\section{\bf Classification of $\mb{Z}^{\times 2}$-graded
AS-regular algebras with 3 generators}
\label{xxsec3}

We continue to assume the notation and hypotheses of the preceding
section, so that $A$ is a $\mb{Z}^{\times 2}$-graded AS-regular
algebra with $3$ generators which is a domain and not a normal extension.

\subsection{Case $\ell \neq 0$.}
\label{xxsec3.1}
We assume now that $\ell \neq 0$, so the leading term of $r_4$ is
$x_3x_1^2$.  Then we can take $\ell = -1$ and by the reductions in
the previous section, we assume that $A$ is presented by the
following four minimal relations:
\begin{align*}
r_1: x_2x_1 & = p x_1x_2 + m x_1^2  \qquad (p \neq 0\
\text{and}\ m = 0,\ \text{or}\ p = 1 = m) \\
r_2: x_3x_2 & = ax_3 x_1 + n x_1x_3 +bx_2x_3 \\
r_3: x_3^2x_1 & = c x_1x_3^2 + d x_2 x_3^2 + ex_3 x_1 x_3 \\
r_4: x_3x_1^2 & = (f x_1^2  + g x_1 x_2  + h x_2^2 )x_3
+ (j x_1 + k x_2) x_3 x_1.
\end{align*}
Recall that the overlap between $r_1$ and $r_2$ also gives the relation
$$r_5: x_3x_1x_2 =qn x_1x_3x_1 + qb x_2x_3x_1 + z x_3x_1^2,$$
where $q = p^{-1}$ and $z = qa - qm$.

\begin{lemma}
\label{xxlem3.1}
The set of monomials $\{ x_1^i x_2^j (x_3x_1)^k x_3^l
\mid i,j,k,l\geq 0\}$ is a $k$-basis for $A$, and all overlap
ambiguities among $r_1-r_5$ resolve.
\end{lemma}

\begin{proof}
Recall that since $A$ is AS-regular, the Hilbert series of $A$ is
$h_A(t)=\displaystyle \frac{1}{(1-t)^3(1-t^2)}$.  The set of
monomials not containing a leading term of any relation $r_1-r_5$
is $\{ x_1^i x_2^j (x_3x_1)^k x_3^l \mid i,j,k,l\geq 0\}$, and so
this spans $A$ as a $k$-vector space.  Since this set of monomials
already has the correct Hilbert series, it is immediate from
Bergman's diamond lemma that these monomials are a $k$-basis and
so all overlaps resolve.
\end{proof}

Given the lemma above, as a necessary condition for AS-regularity
we seek conditions on the coefficients of the $r_i$ for which all
overlaps resolve.  We use a simple program in Maple to help us
calculate the coefficients of the new relations that these overlaps produce.

There are three overlaps to be resolved, $x_3^2x_1x_2$ between $r_3$
and $r_5$, $x_3^2x_1^2$ between $r_3$ and $r_4$, and $x_3x_1x_2x_1$
between $r_5$ and $r_1$.  The respective three new relations are the
following:
\begin{gather*}
r_6:  \qquad (-qba-qn+ea-ze)x_3x_1x_3x_1 \\
+(en j+cae+cba+cn-zce+ebzj)x_1x_3x_1x_3\\
+(enk+dba+dae+dn-zde+ebzk)x_2x_3x_1x_3 \\
+(enf+cbn+ac^2-zc^2+dacm+dbnm-qbnc-qb^2cm+ebzf-zdcm)x_1^2x_3^2  \\
+(eng+cad+dacp+dbnp-qbnd-zcd+ebzg-zdcp)x_1x_2x_3^2 \\
+(enh-qb^2d-zd^2+ad^2+db^2+ebzh)x_2^2x_3^2 =0,
\end{gather*}
\begin{gather*}
r_7: \qquad (-ka+e-j)x_3x_1x_3x_1 \\
+(-fj+ce-gzj-gqn-hna-enk-haqn-hazj)x_1x_3x_1x_3 \\
+(de-ebk-gzk-gqb-fk-hba-haqb-hazk)x_2x_3x_1x_3 \\
+(-hn^2+dcm+c^2-gzf-f^2-knc-hazf-hbnm-kbcm)x_1^2x_3^2 \\
+(dcp+cd-fg-zg^2-hnb-knd-hazg-hbnp-kbcp)x_1x_2x_3^2 \\
+(-hb^2+d^2-gzh-fh-azh^2-kbd)x_2^2x_3^2=0,
\end{gather*}
\begin{gather*}
r_8: \qquad (-zf+mj^2+mf+pzj^2 +pfa \\
-jqn-zj^2-kzjm+jn
+kjm^2+knm+pkzjm-qbjm)x_1^2x_3x_1 \\
+(-zg+mg+pga-jzk+mjk-qnk+knp+kzjp^2+mkjp)x_1x_2x_3x_1 \\
+(mh+mk^2-zh+pha-kqb-zk^2+kb+pzk^2)x_2^2x_3x_1 \\
+(pfn-jzf+hnpm^2-kzfm+pgnm+pjzf+mjf \\
+kfm^2-qnf+hnp^2m^2+kzfp^2m+kfpm^2-qbfpm-qbfm)x_1^3x_3 \\
+(gnp^2+pfb-gzj+hnmp^2-kzfp^2-kzgm+pjzg+mjg \\
+kgm^2-gqn+hnmp^3+kzfp^3+pkzgm+mkfp^2-bfp-qbgm)x_1^2x_2x_3 \\
+(hnp^3+pgb-jzh-kzgp+pjzh+mjh-qnh+kzgp^2+mkgp-qbgp)x_1x_2^2x_3 \\
+(phb-kzh+pkzh+mkh-qbh)x_2^3x_3 = 0.
\end{gather*}

We conclude that if $A$ is AS-regular then all of the coefficients
of $r_6, r_7$, and $r_8$ are $0$.  We will use Maple to help solve
for the possible solutions to the system of equations given by
setting the coefficients of $r_6, r_7, r_8$ to $0$, together with
$pq=1$, $z=qa-qm$.  It is helpful to break this up into three cases,
where in each case we can eliminate some parameters to make the
solution simpler.

\subsection{Case $\ell \neq 0$, $m = 0, p = 1$.}
\label{xxsec3.2}
This is the special case of the quantum-type $r_1$ relation where
$x_1$ and $x_2$ commute.  In this case, we are free to make any
change of basis in $kx_1 + kx_2$ without affecting $r_1$.  Replacing
$x_2$ by $x_2 - ax_1$, we can assume that $r_2$ has the form
$x_3x_2 = n x_1x_3 + bx_2x_3$.  If $n=0$, then $x_2$ is a normal
element and we exclude this case.  Thus we may assume that $n \neq 0$;
replacing $x_1$ by $n x_1 + bx_2$, our relation becomes $x_3x_2 = x_1x_3$.
In conclusion, we may assume that $a = b = 0$ and $n = 1$.

Thus we seek to solve the system given by setting the coefficients of
$r_6-r_8$ to $0$, under the conditions $m = 0, p= q = 1, z= a = b = 0,
n = 1$.   It is easy to see without a computer that there are no
solutions: the coefficient of $x_3x_1x_3x_1$ in $r_6$ is
$-qba-ze-qn+ea$, which in this case is equal to $-1$.

\subsection{Case $\ell \neq 0$, $m = 0, p \neq 1$.}
\label{xxsec3.3}
This is the generic case of the quantum-type $r_1$ relation, and is
the case in which we will find almost all of the examples of
AS-regular algebras.  We make some reductions to the form of $r_2$.
If $n = 0$, then $a \neq 0$ (otherwise else $x_2$ is normal) and
$b \neq 0$ (or else $A$ is not a domain).  Thus passing to the
opposite ring if necessary we may assume $n \neq 0$. Further,
replacing $x_1$ by $n x_1$, we may assume that the relation has
the form $r_2: x_3 x_2 = ax_3 x_1 + x_1x_3 +bx_2x_3$.

Before solving the system of equations given by setting the
coefficients of $r_6-r_8$ equal to $0$, we give a series of lemmas
that will be helpful to simplify our analysis.  First, some of
the solutions will lead to non-domains.  The following lemma
will help us avoid these outright.

\begin{lemma}
\label{xxlem3.2}
Suppose that $n = 1$.  If $A$ is a domain, then $d \neq bc$.
\end{lemma}
\begin{proof}
The relation $r_2$ gives $x_1x_3 = x_3x_2 - ax_3x_1 - bx_2x_3$.
Substituting this into the $x_1x_3^2$ term of $r_3$ gives
$x_3^2x_1 = cx_3x_2x_3 + (d-bc) x_2 x_3^2 + (e-ac)x_3 x_1 x_3$.
If $d=bc$, then we see that $x_3(x_3x_1 - cx_2x_3 + (ac-e)x_1x_3) = 0$
is a relation.  As $A$ is a domain, this forces the relation
$x_3x_1 - cx_2x_3 + (ac-e)x_1x_3 = 0$.  However, $r_2$ is the
only relation of degree $(1,1)$ by Lemma~\ref{xxlem2.1}, and
this is a contradiction.
\end{proof}

Next, we note that we can twist away the coefficient $b$.
\begin{lemma}
\label{xxlem3.3} Let $A = A(b) = A \langle x_1, x_2, x_3\rangle/
(r_1, r_2, r_3, r_4)$,  where $b \neq 0$ and
\begin{align*}
r_1: x_2x_1 & = p x_1x_2  \\
r_2: x_3x_2 & = (a/b) x_3 x_1 + x_1x_3 +bx_2x_3 \\
r_3: x_3^2x_1 & = b^2 c x_1x_3^2 + b^3 d x_2 x_3^2 + b e x_3 x_1 x_3 \\
r_4: x_3 x_1^2 & = b^2 f x_1^2x_3  + b^3 g x_1 x_2x_3  +
b^4 h x_2^2x_3 + b j x_1x_3x_1 + b^2 k x_2 x_3 x_1.
\end{align*}
Then $A(b)$ is isomorphic to a graded twist of $A(1)$.
\end{lemma}

\begin{proof}
The ring $A(1)$ has a automorphism $\phi$ defined by $x_1 \to x_1,
x_2 \to x_2, x_3 \to b^{-1}x_3$. If $B$ is the graded twist of
$A(1)$ by this automorphism, then the relations of $B$ can be found
by multiplying each term in the relations of $A(1)$ by an
appropriate power of $b$. A change of variable replacing $x_1$ by
$bx_1$ can then be checked to produce the relations of $A(b)$.
\end{proof}

For some of our examples, we will have to resort to giving an explicit
free resolution of $k$ to prove that the example has finite global
dimension.  The following lemma will cover those cases.

\begin{lemma}
\label{xxlem3.4}
Suppose that $\{ x_1^i x_2^j (x_3x_1)^k x_3^l \mid i,j,k,l\geq 0\}$ is
a $k$-linear basis of $A$, and that there is a complex of left modules
the form
\begin{equation}
\label{E3.4.1}\tag{E3.4.1}
0\to A(-5)\xrightarrow{d_4} A(-4)^{\oplus 3} \xrightarrow{d_3}
A(-3)^{\oplus 2}\oplus A(-2)^{\oplus 2} \xrightarrow{d_2}
A(-1)^{\oplus 3} \xrightarrow{d_1} A \xrightarrow{d_0}{_Ak}\to 0,
\end{equation}
where the elements of the free modules are row vectors, and each
$d_i$ is given by right multiplication by a matrix. We assume that
\begin{enumerate}
\item[(i)]
$d_4 = (x_1,x_2,x_3)$;
\item[(ii)]
$d_1 = (x_1,  x_2,  x_3)^t$ and $d_2$ is constructed from the
relations $r_1-r_4$; and
\item[(iii)]
$d_3$ is given by a matrix of the form
$M := \begin{pmatrix} \alpha x_3^2     & *     &  *  &  \beta x_3\\
                   \phi x_3^2  & *     &  *  &  \varphi x_3 \\
                 -x_3x_1+y x_3& *     &  *  & *
\end{pmatrix}$,
where $\alpha \varphi-\beta\phi \neq 0$ and $y\in kx_1+kx_2+kx_3$.
\end{enumerate}
Then \eqref{E3.4.1} is exact and $A$ has global dimension $4$.
\end{lemma}

\begin{proof}
It is clear from the form of the $k$-basis for $A$ that $x_3$ is a
right nonzerodivisor; thus $d_4$ is injective. We claim that to
prove \eqref{E3.4.1} is exact, we need only show that
$\ker d_3=\im d_4$.  If this equation holds, then there is only one
possible nonzero cohomology of the complex $X$ given by
\eqref{E3.4.1}, namely $H^2(X)$. Since $X$ is a complex, the
Hilbert series of the terms satisfy
$\sum_{i} (-1)^i h_{H^i(X)}(t)= \sum_i (-1)^i h_{X^i}(t)=0$, where
the latter equality follows from the assumed Hilbert series of $A$.
This implies that $H^2(X)=0$ and hence $X$ is exact, proving the claim.

Now suppose that $(v_1, v_2, v_3)\in \ker d_3$. Then
$\alpha v_1 x_3^2+ \phi v_2 x_3^2+ v_3 (-x_3x_1+yx_3)=0$. Express
$v_3$ as a linear combination of $x_1^i x_2^j (x_3x_1)^k x_3^l$.
Then $\alpha v_1 x_3^2+ \phi v_2 x_3^2+ v_3 (-x_3x_1+yx_3)=0$ implies
 that $v_3= wx_3$ for some $w$.  Working modulo $\im d_4$ and
replacing $(v_1,v_2,v_3)$ by $(v_1-wx_1,v_2-wx_2,v_3-wx_3)$, we may
assume that $v_3=0$. Therefore we have $\alpha v_1 x_3^2+ \phi v_2
x_3^2=0$. Since $x_3$ is a right nonzerodivisor, $\alpha v_1+\phi
v_2=0$.  Similarly, by the fourth column of the equation $(v_1, v_2,
v_3) M=0$, we have $\beta v_1+\varphi v_2=0$.  Since $\alpha
\varphi-\beta\phi \neq 0$, we have $v_1= v_2=0$. This shows that
$(v_1,v_2,v_3)\in \im d_4$. It follows that $\ker d_3 = \im d_4$ and
so $X$ is exact by the argument above.  Since $X$ is the free
resolution of the trivial module $k$, $A$ has global dimension four.
\end{proof}

We are now ready to solve the system given by setting the
coefficients of $r_6-r_8$ equal to $0$, under the conditions $m = 0,
p \neq 1, n = 1$.  By Lemma~\ref{xxlem3.2}, since we want $A$ to be a
domain, we can also assume that $d \neq bc$.  By
Lemma~\ref{xxlem3.3}, if $b \neq 0$ then some twist-equivalent
algebra has $b = 1$.  Thus it suffices to assume that $b = 0$ or $b
=1$.  In the latter case (which holds for all solutions we find) we
will twist the $b$ back and include it as a parameter when we
present the solutions. Solving the system with the constraints
listed above gives $8$ families of solutions, all of which turn out
to be AS-regular, and which we list in the following pages along
with some of their important properties.  We state many of the
properties of these algebras without further proof.  Any claims
about twist equivalence follow from Lemma~\ref{xxlem3.3}, and the
other properties are generally checked by a tedious but
straightforward repeated application of the relations. See
Remark~\ref{xxrem3.14} for more about our computational methods.

\begin{example}
\label{xxex3.5}
Let ${\mathcal A}(b,q)$ be the algebra with the relations
\begin{align*} r_1: x_2x_1 & = \frac{1}{q}x_1x_2 \ \ (q \neq 1) \\
r_2: x_3x_2 & = -\frac{1}{q^2b} x_3 x_1 + x_1x_3 +bx_2x_3  \ \ (b \neq 0) \\
r_3: x_3^2x_1 & = -q^3b^2 x_1x_3^2 + (q^2 + q)bx_3 x_1 x_3 \\
r_4:  x_3 x_1^2 & = -q^3b^2 x_1^2x_3 + (q^2 + q)bx_1x_3x_1.
\end{align*}
\begin{enumerate}
\item[(3.5.1)]
${\mathcal A}(b,q)$ is isomorphic to a graded twist of ${\mathcal A}(1,q)$
[Lemma \ref{xxlem3.3}].
\item[(3.5.2)]
The opposite ring of ${\mathcal A}(b,q)$ is isomorphic to
${\mathcal A}(1/b,1/q)$.
\item[(3.5.3)]
There is a graded algebra isomorphism $\sigma:{\mathcal A}(b,q)\to
{\mathcal A}(q^2b,q^{-1})$ determined by
$$\sigma: x_1\to b x_2, x_2\to q^{-2}b^{-1} x_1, x_3\to x_3.$$
\item[(3.5.4)]
The element $y:=x_3x_1-q^2 bx_1x_3$ is normal in ${\mathcal A}(b,q)$
and ${\mathcal A}(b,q)/(y)$ is an AS-regular algebra of dimension 3.
\end{enumerate}

Let $\mc{A} = {\mathcal A}(b,q)$ for some fixed $b,q$ and let $y$ be
the normal element given above. Because we know that $h_{\mc{A}}(t)
= \displaystyle \frac{1}{(1-t)^3(1-t^2)}$  by construction and
$h_{\mc{A}/(y)}(t) = \displaystyle \frac{1}{(1-t^3)}$, the normal
element $y$ must be regular. Thus $\mc{A}$ satisfies the condition
in Lemma~\ref{xxlem1.3}(c); in particular, it is AS-regular and has
the other good properties listed there.

The algebra $\mc{A}$ is also an iterated Ore extension.   Let $B$ be
the subalgebra generated by $x_1, x_3$ and $y=x_3x_1-q^2bx_1x_3$.
Then
$$B=k\langle x_1,y, x_3 \rangle/(yx_1-qb x_1y,
x_3y-qbx_3y,x_3x_1-q^2bx_1x_3-y)$$ is an iterated Ore extension,
hence it is AS-regular of dimension 3. Now one may check that
$\mc{A} \cong B[x_2; \phi, \delta]$, where $\displaystyle \sigma:
x_1 \to \frac{1}{q} x_1, x_3 \to \frac{1}{b} x_3$ and $\displaystyle
\delta: x_1 \to 0, x_3 \to \frac{1}{q^2 b^2} x_3 x_1-\frac{1}{b} x_1
x_3$. Thus Lemma~\ref{xxlem1.3}(a) also applies to the algebra
$\mc{A}$.
\end{example}

\begin{example}
\label{xxex3.6} Let ${\mathcal B}(b)$ denote the algebra with
relations
\begin{align*} r_1: x_2x_1 & = - x_1x_2 \\
r_2: x_3x_2 & = -\frac{1}{b}x_3 x_1 + x_1x_3 + bx_2x_3 \ \ (b \neq 0)\\
r_3: x_3^2x_1 & = b^2 x_1x_3^2  \\
r_4:  x_3 x_1^2 & = -b^3 x_1 x_2x_3 +b x_1x_3x_1 + b^2x_2 x_3x_1.
\end{align*}
\begin{enumerate}
\item[(3.6.1)]
${\mathcal B}(b)$ is isomorphic to a graded twist of ${\mathcal B}(1)$.
\item[(3.6.2)]
The opposite ring of ${\mathcal B}(b)$ is isomorphic to ${\mathcal B}(1/b)$.
\item[(3.6.3)]
There is a graded algebra automorphism of ${\mathcal B}(b)$ determined by
$$\sigma: x_1\to b x_2, x_2\to b^{-1}x_1, x_3\to x_3.$$
\item[(3.6.4)]
The element $y:=x_3x_1-b^2 x_2x_3$ is a normal element, and the factor
algebra ${\mathcal B}(b)/(y)$ is a 3-dimensional AS-regular algebra
generated by 3 elements.   The element $y$ is regular by the same
proof as for the family $\mc{A}$, and so ${\mathcal B}(b)/(y)$
satisfies Lemma \ref{xxlem1.3}(c). Therefore ${\mathcal B}(b)$ has
the properties listed in Lemma \ref{xxlem1.3}.
\end{enumerate}

\end{example}

\begin{example}
\label{xxex3.7} 
Let ${\mathcal C}(b)$ denote the algebra with relations
\begin{align*}r_1: x_2x_1 & = - x_1x_2 \\
r_2: x_3x_2 & = -\frac{1}{b}x_3 x_1 + x_1x_3 + bx_2x_3 \ \ (b \neq 0) \\
r_3: x_3^2x_1 & = b^2 x_1x_3^2  \\
r_4:  x_3 x_1^2 & = b^3 x_1 x_2x_3 +b x_1x_3x_1 + b^2x_2 x_3x_1.
\end{align*}
\begin{enumerate}
\item[(3.7.1)]
${\mathcal C}(b)$ is isomorphic to a graded twist of ${\mathcal C}(1)$.
\item[(3.7.2)]
The opposite ring of ${\mathcal C}(b)$ is isomorphic to ${\mathcal C}(1/b)$.
\item[(3.7.3)]
There is a graded algebra automorphism of ${\mathcal C}(b)$ determined by
$$\sigma: x_1\to b x_2, x_2\to b^{-1}x_1, x_3\to x_3.$$
\item[(3.7.4)]
The element $y:=x_3x_1-bx_1x_3$ is a normal element, and the factor
algebra ${\mathcal C}(b)/(y)$ is a 3-dimensional AS-regular algebra
generated by 3 elements.  Similarly as above, $y$ is a regular
element and Lemma~\ref{xxlem1.3}(c) holds.
\end{enumerate}
\end{example}

\begin{example}
\label{xxex3.8} Let ${\mathcal D}(b,h)$ be the algebra with
relations
\begin{align*} r_1: x_2x_1 & = -x_1x_2 \\
r_2: x_3x_2 & = -\frac{1}{b}x_3 x_1 + x_1x_3 +bx_2x_3 \ \ ( b \neq 0)\\
r_3: x_3^2x_1 & = b^2 x_1x_3^2 \\
r_4:  x_3 x_1^2 & = \bigg(\frac{h}{b^2} - b^2\bigg)x_1^2x_3  + h x_2^2 x_3.
\end{align*}
\begin{enumerate}
\item[(3.8.1)]
${\mathcal D}(b,h)$ is isomorphic to a graded twist of
${\mathcal D}(1, hb^{-4})$.
\item[(3.8.2)]
The opposite ring of ${\mathcal D}(b,h)$ is isomorphic to
${\mathcal D}(1/b,hb^{-8})$.
\item[(3.8.3)]
There is a graded algebra isomorphism
$\sigma:{\mathcal D}(b,h)\to {\mathcal D}(b,2b^4-h)$ determined by
$$\sigma: x_1\to b x_2, x_2\to b^{-1}x_1, x_3\to x_3.$$
\end{enumerate}

We need further analysis to see why the algebra is AS-regular. Following
Lemma \ref{xxlem1.3}(b), we may assume that $b = 1$. We construct a
potential free resolution of $\mc{D}:=\mc{D}(1,h)$ of the form
$$0\to {\mathcal D}(-5)\xrightarrow{d_4} {\mathcal D}(-4)^{\oplus 3}
\xrightarrow{d_3} {\mathcal D}(-3)^{\oplus 2}\oplus
{\mathcal D}(-2)^{\oplus 2} \xrightarrow{d_2} {\mathcal D}(-1)^{\oplus 3}
\xrightarrow{d_1} {\mathcal D} \xrightarrow{d_0}{_{\mathcal D}k} \to 0$$
where here
\begin{gather*}
d_3: \begin{pmatrix} x_3^2 &  -x_1x_3 & -x_2  &  x_3\\
                 0  & 0 & - x_1  &  x_3\\
- x_3x_1 & (h-1) x_1^2+h x_2^2 &0 & -x_1- x_2
\end{pmatrix},  \\
d_2:
\begin{pmatrix} -x_2   &   -x_1   &   0   \\
                 -x_3   &   -x_3   &x_1+x_2\\
                 -x_3^2 &     0    &x_1x_3 \\
   -x_3x_1&     0    &(h-1) x_1^2+ h x_2^2
\end{pmatrix},
\end{gather*}
and the other maps are as in Lemma~\ref{xxlem3.4}.  It is
straightforward to check that this is a complex, and it is then a
free resolution of $k$ by Lemma~\ref{xxlem3.4}.  Thus $\mc{D}$ has
global dimension $4$.

We also need to show that $\mc{D}$ has enough normal elements. It
follows from the relations $r_1$ and $r_2$ that $x_1^2+x_2^2$ is
central in ${\mathcal D}$, and it follows from the relations $r_2$
and $r_3$ that $x_3^2$ is central. Using the relations $r_1$ and $r_4$,
one sees that $x_1^2$ is central in $B:={\mathcal D}/(x_1^2+x_2^2, x_3^2)$.
Let $C$ be the factor ring $B/(x_1^2)$. Then $x_3x_1-x_1x_3$ is normal
in $C$. Finally, given the leading terms we have, one may check that
$C/(x_3x_1-x_1x_3)$ is spanned by
$\{1, x_1, x_2, x_3, x_1x_2,x_1x_3, x_2x_3, x_1x_2x_3\}$, and so is
finite-dimensional.  Lemma \ref{xxlem1.3}(d) now applies to $\mc{D}$,
and using also Lemma \ref{xxlem1.3}(b), $\mc{D}(b,h)$ is strongly
noetherian, Auslander regular and Cohen-Macaulay.
\end{example}

\begin{example}
\label{xxex3.9} Let ${\mathcal E}(b, \gamma)$ be the algebra with
relations
\begin{align*} r_1: x_2x_1 & = -x_1x_2 \\
r_2: x_3x_2 & = -\frac{1}{b}x_3 x_1 + x_1x_3 +bx_2x_3  \ \ ( b \neq 0) \\
r_3: x_3^2x_1 & = b^3 x_2 x_3^2  \\
r_4:  x_3 x_1^2 & = \gamma b^3 x_1 x_2x_3 + b x_1x_3x_1 +
b^2x_2x_3x_1 \ \ (\gamma = \pm \sqrt{-1}.)
\end{align*}
\begin{enumerate}
\item[(3.9.1)]
${\mathcal E}(b, \gamma)$ is isomorphic to a graded twist of
${\mathcal E}(1, \gamma)$.
\item[(3.9.2)]
The opposite ring of ${\mathcal E}(b, \gamma)$ is isomorphic to
${\mathcal E}(b^{-1}, \gamma^{-1})$.
\item[(3.9.3)]
There is a graded algebra automorphism of ${\mathcal E}(b, \gamma)$
determined by
$$\sigma: x_1\to b x_2, x_2\to b^{-1}x_1, x_3\to x_3.$$
\end{enumerate}

In this and the remaining examples in this case, the proof of
AS-regularity is by the same method as in Example \ref{xxex3.8}:  we
twist away $b$ and then construct an explicit free resolution of $k$
and a sequence of normal elements.  Let ${\mathcal E}:={\mathcal
E}(1, \gamma)$.  One may check using Lemma~\ref{xxlem3.4} that
$_{\mathcal E}k$ has a free resolution of the form given in that
lemma, where
\begin{gather*}
d_3: \begin{pmatrix} -x_3^2 &  x_2x_3 & - x_2  &  -x_3 \\
                 0  & 0 & - x_1  &  - x_3\\
x_1x_3+ x_2x_3- x_3x_1 &\gamma x_1x_2 &0 & x_1+ x_2
\end{pmatrix}, \\
d_2:
\begin{pmatrix} -x_2   &   -x_1   &   0   \\
                 -x_3   &   -x_3   &x_1+x_2\\
                 -x_3^2 &     0    &x_2x_3 \\
   -x_3x_1+x_1x_3+x_2x_3&     0    &\gamma x_1x_2
\end{pmatrix}.
\end{gather*}
Further, ${\mathcal E}$ has normal elements $x_1x_2, x_3^2,
x_1^2+x_2^2$. It is easily checked that $B:={\mathcal E}/ (x_1x_2,
x_3^2, x_1^2+x_2^2)$ has a normal element $x_3x_1-x_2x_3$ and that
$B/(x_3x_1-x_2x_3)$ is finite dimensional.  Now Lemma
\ref{xxlem1.3}(b,d) applies to show $\mc{E}(b, \gamma)$ is strongly
noetherian, Auslander regular and Cohen-Macaulay.
\end{example}

\begin{example}
\label{xxex3.10} Let ${\mathcal F}(b, \gamma)$ denote the algebra
with relations
\begin{align*} r_1: x_2x_1 & = \gamma^2 x_1x_2 \ \
(\gamma = \text{primitive}\ \sqrt[3]{1}) \\
r_2: x_3x_2 & = -\frac{\gamma}{b}x_3 x_1 + x_1x_3 +bx_2x_3 \ \
(b \neq 0)  \\
r_3: x_3^2x_1 & = \gamma^2 b^3 x_2 x_3^2  -bx_3 x_1 x_3 \\
r_4:  x_3 x_1^2 & =\gamma b^3 x_1 x_2x_3 + \gamma^2 b^2 x_2x_3x_1.
\end{align*}
\begin{enumerate}
\item[(3.10.1)]
${\mathcal F}(b, \gamma)$ is isomorphic to a graded twist of
${\mathcal F}(1, \gamma)$.
\item[(3.10.2)]
The opposite ring of ${\mathcal F}(b, \gamma)$ is isomorphic to
${\mathcal F}(b^{-1}, \gamma^{-1})$.
\end{enumerate}
As in the previous examples, we let ${\mathcal F}={\mathcal F}(1,\gamma)$,
and there is a free resolution of $_{\mathcal F}k$  satisfying
Lemma~\ref{xxlem3.4} with
\begin{gather*}
d_3: \begin{pmatrix} \gamma^2 x_3^2 & -\gamma x_3x_1+x_2x_3&-\gamma x_2 &  x_3\\
                 0  & 0 & \gamma^2 x_1  &  x_3\\
-\gamma^2 x_2x_3+ x_3x_1 & x_1x_2 &0 & -\gamma x_1- x_2
\end{pmatrix}, \\
d_2: \begin{pmatrix} -x_2   &   \gamma^2 x_1   &   0   \\
                 -\gamma x_3   &   -x_3   &x_1+x_2\\
                 -x_3^2 &     0    &\gamma^2 x_2x_3-x_3x_1 \\
\gamma^2 x_2x_3-x_3x_1 &     0    & \gamma x_1x_2
\end{pmatrix}.
\end{gather*}
The algebra ${\mathcal F}$ has no normal elements of degree 1 and 2,
but has three normal elements of degree 3: $x_1^3+x_2^3, x_3^3$ and
$x_1^2x_2$. Let $B:={\mathcal F}/(x_1^3+x_2^3, x_3^3, x_1^2x_2)$.
Then $B$ has a normal element
$$d=(x_3x_1)^3+x_1(x_3x_1)^2x_3-\gamma^2x_2(x_3x_1)^2x_3
-x_1^2(x_3x_1)x_3^2-x_1x_2(x_3x_1)x_3^2$$ and $B/(d)$ is finite
dimensional over $k$.  Now Lemma \ref{xxlem1.3}(b,d) applies to show
$\mc{F}(b, \gamma)$ is AS-regular with the usual good properties.
\end{example}

\begin{example}
\label{xxex3.11} Let ${\underline{\mathcal F}}(b, \gamma)$ be the
algebra with relations
\begin{align*} r_1: x_2x_1 & = \gamma^2 x_1x_2 \ \
(\gamma = \text{primitive}\ \sqrt[3]{1})\\
r_2: x_3x_2 & = -\frac{\gamma}{b}x_3 x_1 + x_1x_3 +bx_2x_3 \ \ (b \neq 0) \\
r_3: x_3^2x_1 & = \gamma^2 b^2x_1x_3^2 -b^3x_2 x_3^2 -bx_3 x_1 x_3 \\
r_4:  x_3 x_1^2 & = -b^4 x_2^2 x_3 + \gamma^2 b x_1x_3x_1 -b^2x_2x_3x_1.
\end{align*}
\begin{enumerate}
\item[(3.11.1)]
${\underline{\mathcal F}}(b, \gamma)$ is isomorphic to a graded twist
of ${\underline{\mathcal F}}(1, \gamma)$.
\item[(3.11.2)]
The opposite ring of ${\underline{\mathcal F}}(b, \gamma)$ is
isomorphic to  ${\underline{\mathcal F}}(b^{-1}, \gamma^{-1})$.
\item[(3.11.3)]
There is a graded algebra isomorphism $\sigma:{\mathcal F}(b, \gamma)
\to {\underline{\mathcal F}}(\gamma^2 b, \gamma^2)$ determined by
$$\sigma: x_1\to b x_2, x_2\to \gamma b^{-1} x_1, x_3\to x_3.$$
\end{enumerate}
The isomorphism above shows why we have named this family
$\underline{\mc{F}}(b, \gamma)$.  Since the previous example proved
that $\mc{F}(b, \gamma)$ has all desired properties, we need not
study this example further.
\end{example}

\begin{example}
\label{xxex3.12} Let ${\mathcal G}(b, \gamma)$  denote the algebra
with relations
\begin{align*} r_1: x_2x_1 & = -x_1x_2 \\
r_2: x_3x_2 & = -\frac{1}{b}x_3 x_1 + x_1x_3 +bx_2x_3  \ \ (b \neq 0)\\
r_3: x_3^2x_1 & = b^3 x_2 x_3^2  \\
r_4:  x_3 x_1^2 & = \frac{b^2}{2\gamma} x_1^2x_3 +
\gamma b^4 x_2^2x_3 \ \ (\gamma = \frac{1 \pm i}{2}).
\end{align*}
\begin{enumerate}
\item[(3.12.1)]
${\mathcal G}(b, \gamma)$ is isomorphic to a graded twist of
${\mathcal G}(1, \gamma)$.
\item[(3.12.2)]
The opposite ring of ${\mathcal G}(b, \gamma)$ is isomorphic to
${\mathcal G}(b^{-1}, \overline{\gamma})$.
\item[(3.12.3)]
There is a graded algebra automorphism of ${\mathcal G}(b, \gamma)$
determined by
$$\sigma: x_1\to b x_2, x_2\to b^{-1}x_1, x_3\to x_3.$$
\end{enumerate}
Similarly as in the previous examples, Setting ${\mathcal G}=
{\mathcal G}(1,\gamma)$, there is
a free resolution of $_{\mathcal G}k$ satisfying Lemma~\ref{xxlem3.4},
where
$$d_3: \begin{pmatrix} 0      & 0      & -x_2  &  x_3\\
                 x_3^2  &-x_2x_3 & -x_1  &  x_3\\
                 -x_3x_1&\bar{\gamma} x_1^2+\gamma x_2^2 & 0& -x_1-x_2
\end{pmatrix},\ d_2:
\begin{pmatrix} -x_2   &   -x_1   &   0   \\
                 -x_3   &   -x_3   &x_1+x_2\\
                 -x_3^2 &     0    &x_2x_3 \\
                 -x_3x_1&     0    &\bar{\gamma}x_1^2+\gamma x_2^2
\end{pmatrix}.$$
One can check that ${\mathcal G}$ has normal elements $x_3^2,
x_1^2+x_2^2$. Set $B:={\mathcal G}/(x_3^2, x_1^2+x_2^2)$. Then $B$
has a normal element $x_1^2$, and  $C:=B/(x_1^2)$ has normal
elements $x_3x_1-x_1x_3$ and $x_3x_2-x_2x_3$. Finally, the factor
ring $C/(x_3x_1-x_1x_3,x_3x_2-x_2x_3)$ is finite dimensional.   This
is sufficient to show that ${\mathcal G}(b,\gamma)$ is strongly
noetherian, Auslander regular and Cohen-Macaulay by Lemma
\ref{xxlem1.3}(b,d).
\end{example}

\subsection{Case $\ell \neq 0$, $p = m = 1$}
\label{xxsec3.4}
This is the case of the Jordan type $r_1$, $x_2x_1 = x_1x_2 + x_1^2$.
We make a reduction to $r_2$.  Replacing $x_2$ by $ax_1 + x_2$
(which does not change the relation $r_1$), we can assume that $r_2$
 has the form $x_3x_2 = n x_1x_3 + bx_2x_3$.  In other words, we may
assume that $a = 0$.

We use Maple to solve for the conditions that all coefficients of
$r_6, r_7$, and $r_8$ are $0$, assuming that $p=1, m =1, a= 0$.
Clearly we can also assume that we do not have $c=d= 0$, in which
case $r_3$ would lead to a non-domain.  With these constraints
there is only one solution family, as follows.

\begin{example}
\label{xxex3.13} Let ${\mathcal H}(b)$  be the algebra with the four
relations
\begin{align*}r_1: x_2x_1 & = x_1x_2 + x_1^2 \\
r_2: x_3x_2 & = 2b  x_1x_3 + b x_2x_3 \ \  \\
r_3: x_3^2x_1 & = -b^2 x_1x_3^2 + 2 b x_3x_1x_3 \\
r_4:  x_3x_1^2 & =  -b^2 x_1^2x_3 + 2b x_1x_3x_1.
\end{align*}
\begin{enumerate}
\item[(3.13.1)]
${\mathcal H}(b)$ is a graded twist of ${\mathcal H}(1)$ by the
automorphism sending $x_1 \to x_1$, $x_2 \to x_2$, $x_3 \to
b^{-1}x_3$.
\item[(3.13.2)]
The opposite ring of ${\mathcal H}(b)$ is isomorphic to ${\mathcal H}(-b)$.
\item[(3.13.3)]
$\mc{H}(1)$ has a normal element $y:=x_3x_1 -x_1x_3$, and the factor
ring ${\mathcal H}(1)/(y)$ is an AS-regular algebra of dimension 3.
\end{enumerate}
The same argument as in Example~\ref{xxex3.5} shows that $y$ is a
regular element, and so Lemma~\ref{xxlem1.3}(b,c) applies to show
that $\mc{H}(b)$ is AS-regular and has all good properties.  The
example ${\mathcal H}(b)$ can also be written as an Ore extension:
If $B = k\langle x_1, x_3 \rangle/(r_3, r_4)$, then $B$ is a cubic
AS-regular algebra, and ${\mathcal H}(b) \cong B[x_2; \sigma,
\delta]$ for the appropriate $\sigma, \delta$.  Thus
Lemma~\ref{xxlem1.3}(a) also applies to this example.
\end{example}

\subsection{Other leading terms for $r_4$}
\label{xxsec3.5} We still need to consider whether we get any
AS-regular algebras in which the leading term of $r_4$ is smaller
than $x_3x_1^2$.  The answer is no, but we can only prove this
through a computer calculation to exhaust the possibilities.  As the
method is entirely similar as in the preceding cases, we do not give
all of the details of the messy coefficients that arise.  The
interested reader can find more details at the first author's
webpage; see Remark~\ref{xxrem3.14} below.

Suppose first that $\ell = 0$ but that the leading term of $r_4$ is
the next highest possibility $x_2x_3x_1$, so $k \neq 0$.   We can
assume that $k = -1$ and so $r_4$ is
\[
r_4: x_2 x_3 x_1 =(f x_1^2  + g x_1 x_2  + h x_2^2 )x_3 + jx_1x_3x_1,
\]
while the relations $r_1,r_2,r_3,r_5$ are unchanged.   We note that an
easy calculation using the multi-graded Hilbert series \eqref{E2.1.1}
shows that since $A$ is AS-regular we should have $\dim_k A_{(2,2)} = 6$
and $\dim_k A_{(3,1)} = 7$.

Now the leading terms of the relations $r_1-r_5$ already imply that
$A_{(2,2)}$ is spanned by the $6$ monomials
\begin{equation}
\label{E3.13.1}\tag{E3.13.1}
\{x_3x_1x_3x_1, x_3x_1^2 x_3, x_1 x_3 x_1 x_3,
x_2^2 x_3^2, , x_1 x_2 x_3^2, x_1^2 x_3^2\},
\end{equation}
so these are a basis for $A_{(2,2)}$.   The overlaps $x_3^2x_1x_2$
between $r_5$ and $r_3$ and $x_3x_2x_3x_1$ between $r_2$ and $r_4$
produce respective relations $r_6$ and $r_7$ involving monomials in
\eqref{E3.13.1}, and so every coefficient of $r_6$ and $r_7$ is $0$.
We suppress the formulas for these relations, but mention that the
coefficient in $r_6$ of $x_3x_1x_3x_1$ is $-qba-qn+ea-ze$, just as
it was in the $\ell = -1$ case (though $r_6$ as a whole is now
different), so in particular this forces $-qba-qn+ea-ze = 0$.

We also now have two overlaps in multi-degree $(3,1)$ to resolve,
$x_3x_1x_2x_1$ and $x_2x_3x_1x_2$.  It is easy to check that the
first leads to a genuine new relation $r_8$ with leading term
$px_3x_1^2x_2$.  Then the given leading terms of $r_1, r_2, r_4, r_5$
plus the new leading term $x_3x_1^2x_2$ imply that $A_{(3,1)}$ is
spanned by the seven monomials
$$\{x_3x_1^3,   x_1x_3x_1^2, x_1^2x_3x_1, x_2^3x_3,x_1x_2^2x_3,
x_1^2x_2x_3,x_1^3x_3 \},$$
which must be then be a basis for $A_{(3,1)}$.  Resolving
$x_2x_3x_1x_2$ gives a new degree-$(3,1)$ relation $r_9$, every
coefficient of which must now be $0$.

We now have the same three cases for $r_1$ as before, and the same
reductions to $r_2$ in each case are allowed.  If $p=1, m = 0$, we
can reduce $r_2$ so that $a=b=0, n = 1$.  Since the coefficient
$(ea-q-qba-aqe)$ of $r_6$ cannot be $0$, we get no solutions. If $m
= 0, p\neq 1$, then we can assume that $a, b \neq 0$, $n = 1$. Again
in order for $A$ to be a domain, Lemma~\ref{xxlem3.2} shows we can
assume that $d \neq bc$.  Solving for the coefficients of $r_6, r_7,
r_9$ to be zero under these constraints using Maple gives no
solutions.   Finally, if $p=1, m = 1$, we can adjust $r_2$ so that
$a = 0$, and we may also assume that we do not have $c=d=0$ (or else
$A$ will not be a domain).  Setting the coefficients of $r_6, r_7,
r_9$ to $0$, there are again no solutions.

The remaining possibilities for the leading terms of $r_4$ are
comparatively easily eliminated. If we have $\ell = 0, k = 0, h \neq
0$, we can assume that $r_4$ has the form
$$
r_4: x_2^2x_3 =(f x_1^2  + g x_1 x_2)x_3 + jx_1x_3x_1
$$
and the relations $r_1,r_2,r_3,r_5$ are unchanged from the previous
cases. With the leading terms we have, the six monomials
\[
\{x_3x_1x_3x_1, x_3x_1^2 x_3, x_1 x_3 x_1 x_3, x_2x_3x_1x_3, x_1 x_2
x_3^2, x_1^2 x_3^2\}
\]
span $A_{(2,2)}$ and so must be a basis.   One may check that
resolving the overlap $x_3x_2^2x_3$ between $r_4$ and $r_2$
leads to a new degree (2,2)-relation among these monomials
$r_6:  -j x_3x_1x_3x_1 + \dots$, which forces $j = 0$. But then
the relation $r_4$ shows that $A$ is not a domain, since $r_4$
has the form $f x_3$ where $f$ is not a relation.

The final case is where $\ell = k = h = 0$.  In this case $r_4$ has
the form $x_1 f$ for a nonzero $f$ and $A$ cannot be a domain.

\begin{remark}
\label{xxrem3.14} We give some brief comments about the
computational methods which were used to obtain the results above.
Our main program is a simple Maple program which reduces an element
of the free algebra (with coefficients involving unknown parameters)
using given relations. This was used to calculate the relations
$r_6, r_7,$ etc. arising from overlaps in the diamond lemma. The
Maple solve command was used to solve the simultaneous system of
equations given by setting all of the coefficients of these
relations to $0$. The reader might be concerned at the lack of proof
that the Maple solve command really finds all solutions to such a
complicated system.  We did verify the computation in the most
important case by checking that when $\ell \neq 0, p \neq 1, n = 1$,
the system given by the coefficients of $r_6-r_8$ can actually be
solved by hand to give the same set of solutions.

We also wrote programs in Maple to calculate the matrix $d_3$ in the
free resolution of $_A k$ as in Lemma~\ref{xxlem3.4}, and to find
some normal elements of low degree, as was needed for examples
$\mc{D} -\mc{G}$.  The calculation of $d_3$ could presumably have
also been accomplished by trial and error, but it would have been
difficult to find the required sequences of normal elements by hand,
especially in example $\mc{F}$ where one normal element has degree
$6$.  Our program which reduces elements using the relations can
also be used to check the output of these programs, for instance
that the claimed normal elements really are.

The reader wishing to verify our computation can find the Maple code
for these programs, which we make freely available, on the first
author's website:   http://www.math.ucsd.edu/$\sim$drogalsk.
\end{remark}

\section{\bf Proof of the main results}
\label{xxsec4}

Our main theorems all quickly follow from the results of the last section.

\begin{proof}[Proof of Theorem \ref{xxthm0.1}]
The theorem is just a summarization of the classification we did in
the last section, which showed that any properly $\mb{Z}^{\times
2}$-graded AS-regular algebra which is a domain with three
generators is either a normal extension, or else is isomorphic to
one of the examples $\mc{A}(b,q), \mc{B}(b), \cdots, \mc{H}(b)$ (we
may exclude the family $\underline{\mc{F}}$, so there are eight
families here).  Note that although we sometimes passed to the
opposite ring in our reductions, our calculations of the opposite
rings of our examples show that each opposite ring is always
isomorphic to an algebra already on the list (in fact, an example in
the same family.)  The result follows.
\end{proof}

\begin{proof}[Proof of Theorem \ref{xxthm0.2}]
By our classification, if $A$ is a properly ${\mathbb Z}^{\times
2}$-graded noetherian AS-regular algebra with $3$ generators that is
not a normal extension, then again $A$ is isomorphic to one of the
algebras $\mc{A}-\mc{H}$.   In section~\ref{xxsec3} we showed that
each of these algebras (or some graded twist) satisfies at least one
of the conditions (a),(c), or (d) in Lemma~\ref{xxlem1.3}, and so by
that lemma $A$ is also strongly noetherian, Auslander regular and
Cohen-Macaulay.

If $A$ is a normal extension (of a 3-dimensional AS-regular algebra),
then it is strongly noetherian, Auslander regular and Cohen-Macaulay
by Lemma \ref{xxlem1.3}(c).
\end{proof}

\begin{proof}[Proof of Corollary \ref{xxcor0.3}]
By \cite[Proposition 1.4]{LPWZ}, $A$ is generated by either 2, or 3,
or 4 elements. If $A$ is generated by 3 elements, this is Theorem
\ref{xxthm0.2}. If $A$ is generated by 2 elements, the assertion
follows from \cite[Theorems A, B, C]{LPWZ}. If $A$ is generated by
4 elements, the assertion  follows from \cite[Theorems 0.1 and 0.2]{ZZ1}
and Lemma \ref{xxlem1.3}(a).
\end{proof}

\section{\bf Graded isomorphisms and automorphisms}
\label{xxsec5}

In this section, we study graded isomorphisms between the AS-regular
algebras classified above, for generic values of the parameters.
The first task is to show in the next lemma that any such graded
isomorphism has a limited form; in particular, it preserves the bigrading.

\begin{lemma}
\label{xxlem5.1} Suppose that $A$ and $A'$ are any two algebras among
the examples $\mc{A}$-$\mc{G}$ given in section~\ref{xxsec3}, where
$A$ has relations $r_1: x_2x_1 = px_1x_2$ and
$r_2: x_3x_2 = ax_3x_1 + x_1x_3 + bx_2x_3$, and $A'$ has respective
parameters $p', a', b'$ in its first two relations.

Similarly, let $B$ and $B'$ be any two algebras among the family
$\mc{H}$ in section~\ref{xxsec3}, where $B$ has relations
$\widehat{r}_1: x_2x_1 = x_1x_2 + x_1^2$ and $\widehat{r}_2: x_3x_2
= 2 \beta x_1x_3 + \beta x_2x_3$ and $B'$ has parameter $\beta'$ in
its second relation.  Assume that $p, b, p', b', \beta, \beta' \neq
0, 1$ and $a, a' \neq 0, -1$.
\begin{enumerate}
\item
If there is a graded isomorphism $\sigma: A \to A'$, one of the
following two cases must occur:
\begin{enumerate}
\item[(i)]
$\sigma$ is given by $x_1\to \lambda x_1, x_2 \to \lambda x_2,
x_3\to \mu x_3$ for some $\lambda, \mu \in k^{\times}$; or
\item[(ii)]
$\sigma$ is given by $x_1\to \lambda x_2, x_2\to \rho x_1, x_3\to
\mu x_3$ for some  $\lambda, \mu, \rho \in k^{\times}$.
\end{enumerate}
\item
If there is a graded isomorphism $\sigma: B \to B'$, then $\sigma$
is determined by $x_1\to \lambda x_1, x_2\to \lambda x_2, x_3\to \mu
x_3$ for some $\lambda, \mu \in k^{\times}$.
\item
There does not exist a graded isomorphism $\sigma: A \to B$.
\end{enumerate}
\end{lemma}
\begin{proof}
(a) Any graded algebra isomorphism $\sigma: A \to A'$ is determined
by $\sigma(x_i)=\sum_{j=1}^3 a_{ij} x_j$ for an invertible $3\times
3$-matrix $(a_{ij})_{3\times 3}$.  Applying $\sigma$ to the relation
$r_1$ of $A$ and restricting to the degree $(0,2)$-piece gives rise
to the equation $a_{13}a_{23}=pa_{23}a_{13}$. Since $p\neq 1$, we
have $a_{13}a_{23}=0$, which means that either $a_{13}=0$ or
$a_{23}=0$.  We assume that $a_{13} = 0$; the analysis in the other
case is symmetric.  Now applying $\sigma$ to $r_2$ and considering
the degree $(0,2)$-piece gives $a_{33}a_{23} = ba_{23}a_{33}$. Since
$b \neq 0$, we have $a_{33}= 0$ or $a_{23} = 0$.   Suppose that
$a_{33} = 0$.  Then $\sigma(x_1), \sigma(x_3)$ have degree $(1,0)$
in $A'$, and so must span $kx_1 + kx_2$ in $A'$. Thus the relation
$r'_1$ of $A'$ pulls back under $\sigma$ to give some degree 2
relation in $A$ involving only $x_1$ and $x_3$.  Since $A$ has no
such relation, this is a contradiction.  Thus $a_{23} = 0$. We
conclude that $\sigma(x_1), \sigma(x_2) \in kx_1 + kx_2$, and
$a_{33} \neq 0$.

Now considering the relation $\sigma(r_1)$ again and using that
$p, p' \neq 1$, it is easy to prove that either (i)
$\sigma(x_1) = a_{11} x_1, \sigma(x_2) = a_{22} x_2, p = p'$, or
(ii) $\sigma(x_1) = a_{12} x_2, \sigma(x_2) = a_{21} x_1, p' = p^{-1}$.

Consider case (i).   Looking at $\sigma(r_2)$ and restricting to
degree $(2,0)$ gives a relation
$$(a_{31}x_1+a_{32}x_2)(- a a_{11}x_1+ a_{22}x_2)
= (a_{11}x_1+ b a_{22}x_2)(a_{31}x_1+a_{32}x_2).$$ Since $r'_1$ is
the only relation in $A'$ of degree $(2,0)$ and $a \neq -1, b \neq
1$, it is easy to see that $a_{32} = a_{31} = 0$. Finally,
restricting $\sigma(r_2)$ to degree $(1,1)$ gives a relation
$$a_{33}x_3(-a a_{11}x_1 + a_{22} x_2) =
(a_{11} x_1 + b a_{22} x_2)a_{33}x_3,$$
which must be a multiple of $r'_2$. It is easy to check that
this implies that $a_{11} = a_{22}$.

The argument in case (ii) is similar and we leave it to the reader.

(b)  The proof is similar to the proof of (a) and we omit it.

(c)  Suppose that $\sigma: A \to B$ is a graded isomorphism.  The
same argument as in the first paragraph of the proof of (a) still
applies to $\sigma: A \to B$, since this argument used only the form
of the relations of $A$. So $\sigma(x_1), \sigma(x_2) \in kx_1 +
kx_2$. In particular, $\sigma$ restricts to an isomorphism $k
\langle x_1, x_2 \rangle/(r_1) \to k \langle x_1, x_2
\rangle/(\widehat{r}_1)$.  This is impossible, since the quantum
plane and the Jordan plane are well-known to be non-isomorphic; or,
a simple argument similar to the above easily proves this.
\end{proof}

We call an isomorphism of the form given in part (a)(i) or part (b)
of the lemma {\it trivial} and an isomorphism of the form given in
part (a)(ii) {\it quasi-trivial}.  The lemma now makes it easy to
determine all possible isomorphisms and automorphisms of the
algebras classified earlier, when the parameters are not special. To
make this precise, suppose that $A$ is one of the algebras of types
$\mc{A}-\mc{H}$ given in Section~\ref{xxsec3}.  In order to give a
uniform treatment, we include the family $\underline{\mc{F}}$.  We
say that $A$ is \emph{generic} if the parameters in $r_1, r_2$
satisfy the hypothesis of Lemma~\ref{xxlem5.1}.  Explicitly, this is
equivalent to $b \neq 1, q^2b \neq 1$ for type $\mc{A}(b,q)$, to $b
\neq 1, b \neq \gamma$ for type $\mc{F}(b, \gamma)$ and type
$\underline{\mc{F}}(b, \gamma)$, and to $b \neq 1$ for $\mc{H}(b)$
and all of the other types.  Since the point of this section is to
investigate isomorphisms and automorphisms among these algebras, we
consider any two generic AS-regular algebras in the families
$\mc{A}-\mc{H}$ on our list to be the same at the moment if and only
if they have exactly the same relations; if they have different
relations we call them {\it distinct}.

\begin{theorem}
\label{xxthm5.2} Consider the generic AS-regular algebras of
types $\mc{A}-\mc{H}$.
\begin{enumerate}
\item
If $A$ is a generic AS-regular algebra, the graded automorphism
group of $A$ is isomorphic either to $k^{\times} \times k^{\times}$
or to $k^{\times} \times k^{\times} \times {\mathbb Z}/(2)$.  The
first case occurs if $A$ is of type $\mc{A}(b,q)$ with $q \neq -1$,
$\mc{D}(h,b)$ with $h \neq b^4$, $\mc{F}$, $\underline{\mc{F}}$, or
$\mc{H}$. The second case occurs if $A$ is of type $\mc{A}(b, -1),
\mc{B}, \mc{C}, \mc{D}(h,b)$ with $h = b^4$, $\mc{E}$, or $\mc{G}$.
\item
All distinct generic AS-regular algebras are pairwise non-isomorphic
except for the isomorphisms given in (3.5.3), (3.8.3) and (3.11.3).
\end{enumerate}
\end{theorem}

\begin{proof}
Let $A, A'$ be generic AS-regular algebras.   By Lemma~\ref{xxlem5.1}
we see that any graded isomorphism $A \to A'$ is either trivial or
quasi-trivial.  Since $A$ is $\mb{Z}^{\times 2}$-graded, the graded
automorphism group of $A$ includes at least the trivial automorphisms
$x_1 \to \lambda x_1, x_2 \to \lambda x_2, x_3 \to \mu x_3$ for any
$\lambda, \mu \in k^{\times} \times k^{\times}$.   Moreover, if $A$
has any quasi-trivial automorphism $\sigma$, then clearly $\sigma^2$
is trivial.  So the graded automorphism group is isomorphic to
$k^{\times} \times k^{\times}$ if $A$ has no quasi-trivial automorphism,
or to $k^{\times} \times k^{\times} \times {\mathbb Z}/(2)$ if $A$ does
have a quasi-trivial automorphism.

If $A$ is of type $\mc{H}$, it has no quasi-trivial isomorphism to
another generic AS-regular algebra, by Lemma~\ref{xxlem5.1}.  Now
for a given generic AS-regular algebra $A$ of some type
$\mc{A}-\mc{G}$, we claim there is exactly one generic AS-regular
algebra $\wt{A}$ such that there exists a quasi-trivial isomorphism
$A \to \wt{A}$.  In our listing of AS-regular algebras in
Section~\ref{xxsec3}, we gave an explicit such isomorphism $\sigma:
A \to \wt{A}$ in each case (in case $\mc{F}$ or
$\underline{\mc{F}}$, take (3.11.3) or its inverse, respectively.)
Conversely, if $\rho: A \to A'$ is another quasi-trivial
isomorphism, then $\sigma \rho^{-1}: A' \to \wt{A}$ is a trivial
isomorphism. Precomposing with some trivial automorphism $\tau: A'
\to A'$ we can get an isomorphism $\sigma \rho^{-1} \tau: A' \to
\wt{A}$ which sends $x_i \to x_i$ for all $i$.  This is easily seen
to force $A'$ and $\wt{A}$ to have exactly the same relations,
proving the claim. In our listing of the AS-regular algebras in
Section~\ref{xxsec3}, we gave an explicit quasi-trivial automorphism
of the AS-regular algebras $\mc{A}(b,q)$ with $q = -1$, $\mc{B}$,
$\mc{C}$, $\mc{D}(h,b)$ with $h = b^4$, $\mc{E}$, or $\mc{G}$.  Thus
$\wt{A} = A$ for these algebras.  On the other hand, for the other
types $\mc{A}(b,q)$ with $q \neq -1$, $\mc{D}(h,b)$ with $h \neq
b^4$, $\mc{F}$, and $\underline{\mc{F}}$, we gave a quasi-trivial
isomorphism from $A$ to a distinct algebra, and so $\wt{A} \neq A$.
This finishes the proof of part (a).

To prove part (b), note that if $\sigma: A \to A'$ is a trivial
isomorphism between generic AS-regular algebras, then precomposing
with a trivial automorphism we see similarly as in the previous
paragraph that in fact $A'$ and $A$ have exactly the same relations.
On the other hand, if $\sigma: A \to A'$ is a quasi-trivial
isomorphism, then $A' = \wt{A}$ as defined above, and $\wt{A} \neq
A$ exactly for the isomorphisms in (3.5.3), (3.8.3), and (3.11.3).
Note that these isomorphisms match up pairs of examples in family
$\mc{A}$ and pairs of examples in family $\mc{D}$, and match up the
entire family $\mc{F}$ with the family $\underline{\mc{F}}$.
\end{proof}

We note that this final theorem justifies the remark made in the
introduction that the eight families $\mc{A}-\mc{H}$ (excluding
$\underline{\mc{F}}$) in our classification are generically pairwise
non-isomorphic.

\providecommand{\bysame}{\leavevmode\hbox
to3em{\hrulefill}\thinspace}
\providecommand{\MR}{\relax\ifhmode\unskip\space\fi MR }
\providecommand{\MRhref}[2]{%
  \href{http://www.ams.org/mathscinet-getitem?mr=#1}{#2}
} \providecommand{\href}[2]{#2}

\end{document}